\documentclass[11pt]{amsart}
\usepackage{pgf,tikz,pgfplots,color}
\usepackage{mathrsfs} 
\usetikzlibrary{arrows}
\usepackage{fancyhdr}
\usepackage{lastpage} 

\usepackage[margin=1in]{geometry}
\usepackage{geometry}
\usepackage{hyperref}
\hypersetup{
    colorlinks=true,
    linkcolor=blue,
    citecolor=brown,
    filecolor=magenta,
    urlcolor=blue,
}
\usepackage{amsmath}                          
\usepackage{amsfonts}
\usepackage{amssymb}
\usepackage{amsthm}
\usepackage{setspace}
\usepackage{inputenc}
\usepackage{comment}
\usepackage{mathtools}
\usepackage[shortlabels]{enumitem}

\newtheorem{thm}{Theorem}
\newtheorem{cor}[thm]{Corollary} 
\newtheorem{prop}[thm]{Proposition}
\newtheorem{lem}[thm]{Lemma}

\theoremstyle{definition}
\newtheorem{defn}[thm]{Definition}

\newtheorem{note}[thm]{Notation}
\newtheorem{conv}[thm]{Convention}

\newtheorem{innercustomgeneric}{\customgenericname}
\providecommand{\customgenericname}{}
\newcommand{\newcustomtheorem}[2]{%
  \newenvironment{#1}[1]
  {%
   \renewcommand\customgenericname{#2}%
   \renewcommand\theinnercustomgeneric{##1}%
   \innercustomgeneric
  }
  {\endinnercustomgeneric}
}

\newcustomtheorem{customprop}{Proposition}
\newcustomtheorem{customlemma}{Lemma}

\theoremstyle{remark}
\newtheorem{rem}[thm]{Remark}

\newtheorem*{ack}{Acknowledgment}

\def\XXint#1#2#3{{\setbox0=\hbox{$#1{#2#3}{\int}$ }
\vcenter{\hbox{$#2#3$ }}\kern-.6\wd0}}

\newcommand{\Z}{\mathbb{Z}}
\newcommand{\R}{\mathbb{R}} 
\newcommand{\C}{\mathbb{C}}

\newcommand{\N}{\mathbb{N}}

\newcommand{\dbar}{\overline{\partial}}

\newcommand{\Ds}{\mathscr{D}}

\newcommand{\Ec}{\mathcal{E}}
\newcommand{\Es}{\mathscr{E}}
\newcommand{\Fc}{\mathcal{F}}
\newcommand{\Fs}{\mathscr{F}}

\newcommand{\Hc}{\mathcal{H}}

\newcommand{\Pc}{\mathcal{P}}

\newcommand{\Sc}{\mathcal{S}}
\newcommand{\Ss}{\mathscr{S}}
\newcommand{\Tc}{\mathcal{T}}
\newcommand{\Uc}{\mathcal{U}}
\newcommand{\Xs}{\mathscr{X}}
\newcommand{\Ys}{\mathscr{Y}}
\newcommand{\id}{\mathrm{id}}
\newcommand{\eps}{\varepsilon}

\newcommand{\dist}{\operatorname{dist}}
\newcommand{\supp}{\operatorname{supp}}

\newcommand{\Span}{\operatorname{Span}}
\newcommand{\Coorvec}[1]{\frac\partial{\partial#1}}

\newcommand{\Vol}{\operatorname{Vol}}

\title[Sobolev estimates on product domains]{On $\overline\partial$ Homotopy Formulae for Product Domains: Nijenhuis-Woolf's Formulae and Optimal Sobolev Estimates}

\author[]{Liding Yao} 
\address{Department of Mathematics,
	The Ohio State University, Columbus, OH 43210} 
\email{yao.1015@osu.edu}

\author[]{Yuan Zhang} 
\address{Department of Mathematical Sciences,
	Purdue University Fort Wayne, Fort Wayne, IN 46805
} 
\email{zhangyu@pfw.edu}
\pgfplotsset{compat=1.18}
\begin{document}

\begin{abstract}
    We construct   homotopy formulae $f=\overline\partial\mathcal H_qf+\mathcal H_{q+1}\overline\partial f$ for $(0, q)$ forms on the product domain $\Omega_1\times\dots\times\Omega_m$, where each $\Omega_j$ is either a bounded Lipschitz domain in $\C^1$, a bounded strongly pseudoconvex domain with $C^2$ boundary, or a smooth convex domain of finite type. Such homotopy operators $\Hc_q$ yield  solutions to the $\overline\partial$ equation with optimal Sobolev regularity $W^{k,p}\to W^{k,p}$ simultaneously for all $k\in\Z$ and $1<p<\infty$.
\end{abstract}

\maketitle
\section{Introduction}

The goal of this paper is to prove the following:
\begin{thm}\label{Thm::DbarThm}
      Let $\Omega_j\subset\C^{n_j}$ be a bounded Lipschitz domain for each $j=1,\dots, m$, with    $m\ge 1$, such that one of the following holds.
    \begin{itemize}
        \item $\Omega_j\subset\C$ is a planar domain (i.e. $n_j=1$).
        \item $\Omega_j$ is strongly pseudoconvex with $C^2$ boundary or strongly $\C$-linearly convex with $C^{1,1}$ boundary.
        \item $\Omega_j$ is a  smooth convex domain of finite type.
    \end{itemize}
    Let $\Omega:=\Omega_1\times\dots\times\Omega_m$ and $n: = \sum_{j=1}^m n_j$.  
    Then there exist linear operators $\Pc=\Pc^\Omega:\Ss'(\Omega)\to\Ss'(\Omega)$ and $\Hc_q=\Hc_q^\Omega:\Ss'(\Omega;\wedge^{0,q})\to\Ss'(\Omega;\wedge^{0,q-1})$ for $1\le q\le n$, such that
    \begin{enumerate}[(i)]
        \item\label{Item::DbarThm::HT} $f=\Pc f+\Hc_1\dbar f$ for all $f\in\Ss'(\Omega)$; and $f=\dbar\Hc_qf+\Hc_{q+1}\dbar f$ for all $f\in\Ss'(\Omega;\wedge^{0,q})$.
        \item\label{Item::DbarThm::SobBdd} We have Sobolev estimates $\Pc:W^{k,p}(\Omega)\to W^{k,p}(\Omega)$ and $\Hc_q:W^{k,p}(\Omega;\wedge^{0,q})\to W^{k,p}(\Omega;\wedge^{0,q-1})$ for all $k\in\Z$ and $1<p<\infty$.
    \end{enumerate}
\end{thm}
Here $\Ss'(\Omega)$ is the space of distributions on $\Omega$ which admit extension to distributions on $\C^n$, $ W^{k,p}(\Omega)  $ is the Sobolev space on $\Omega$ with  $k\in\Z$ and $1<p<\infty$, and $\Ss'(\Omega;\wedge^{0,q})\  (\text{resp.}\  W^{k,p}(\Omega;\wedge^{0,q})) $  is the space of degree $(0, q)$ forms with coefficients in $\Ss'(\Omega) \  (\text{resp.}\  W^{k,p}(\Omega)  ) $. 
See Notation~\ref{Note::Space::Dist}, Definition ~\ref{Defn::Space::Sob} and Convention \ref{Cov:;spacef} for the precise definitions. We note that $(\Hc_q)_{q=1}^n$ 
do not possess  optimal H\"older estimates. See Remark~\ref{Rmk::Holder}. 

As an immediate consequence of Theorem~\ref{Thm::DbarThm} we obtain a solution operator  to the $\dbar$ equation on product domains for any $(0,q)$ forms with $q\ge1$, together with  the associated Sobolev estimate. Note that in view of a Kerzman-type example (see, e.g. \cite[Example 1]{ZhangSobolevProduct}) the Sobolev regularity is sharp.

\begin{cor}\label{Cor::dbar}
    Let  $\Omega\subset\C^n$ be given as in Theorem~\ref{Thm::DbarThm}. Let $1\le q\le n$. For every $(0,q)$-form $f$ on $\Omega$ whose coefficients are extendable distributions such that $\dbar f=0$ in the sense of distributions, there is a $(0,q-1)$ form $u$ on $\Omega$ whose coefficients are also extendable distributions, such that $\dbar u=f$.
    
    Moreover, for every $k\in\Z$ and $1<p<\infty$ there is a constant  $C>0$, such that if the coefficients of $f$ are in $W^{k,p}$, then we can choose $u$ whose coefficients are in $W^{k,p}$, with the estimate $$\|u\|_{W^{k,p}(\Omega;\wedge^{0,q-1})}\le C\|f\|_{W^{k,p}(\Omega;\wedge^{0,q})}.$$
\end{cor}

The $\dbar$ homotopy formulae play an essential role in  studying the  $\dbar $ problem and have been extensively developed of pseudoconvex domains with certain finite type conditions, using the $\dbar$-Neumann approach (e.g. \cite{GreinerStein1977,FeffermanKohn1988,Chang1989,ChangeNagelStein1992}) and the integral representation approach (e.g. \cite{LiebRange80,Range1990,DFFHolder,GongHolderSPsiCXC2}). We refer the readers to \cite{Range1990,ChenShawBook,LiebMichelBook} for more details. Product domains, owing to their particular structure, fail to have finite type, and  merely admit  Lipschitz boundary regularity. 

The study of the $\dbar$ problem on product domains  was initiated by the work of Henkin  \cite{Henkin}, who established  $L^\infty$ estimates for the $\dbar$ equation on the bidisk using an integral representation by the Cauchy kernel. Landucci  \cite{Landucci} later proved an analogous result   for the canonical solutions. Since the work of Chen-McNeal \cite{ChenMcNeal1}  on    $L^p$ estimates for product domains in $\mathbb C^2$, there has been much progress towards  the  optimal $L^p$ estimates  for $(0,1)$ forms on  Cartesian products of general planar domains. The special case $p=\infty$ (on planar domains) is also called the Kerzman's supnorm estimate problem, posted in \cite{Kerzman71}. The optimal $L^\infty$ estimate was first given by  Fassina-Pan \cite{FassinaPan} for $C^{n-1, \alpha}$ data. Later  Dong-Pan-Zhang \cite{DongPanZhang}   extended this result to   the canonical solutions for continuous data. Kerzman's problem was completely solved recently  by  Yuan  \cite{Yuan} on products of $C^2$ planar domains based on  \cite{DongPanZhang},  and Li \cite{Li} on products of   $C^{1,\alpha}$ planar domains independently. Subsequently Li-Long-Luo \cite{LiLongLuo} further relaxed the boundary regularity of each factor to Lipschitz.   
The   Sobolev  regularity of $\dbar$   was first  investigated by Chakrabarti-Shaw \cite{ChaShaw}  for   the canonical solutions with respect to $(0,1)$ data on  products of certain smooth pseudoconvex domains. In particular, the optimal  $W^{k, p}, k\ge 1$  regularity  was obtained   for    products of smooth planar domains  in $\mathbb C^2$ by Jin-Yuan    \cite{JinYuan}  and in $\mathbb C^n$  by Zhang \cite{ZhangSobolevProduct}. See also      \cite{JakobczakProductSPsiCX, DongLiTreuer, ChenMcNeal2, PanZhang} and the  references therein.   

 In comparison to those results, our theorem allows each factor $\Omega_j\subset\C^{n_j}$ to be non-planar (i.e. we allow $n_j>1$). Such product domains were previously studied in \cite{JakobczakProductSPsiCX,ChaShaw,ChenMcNeal2} for  $(0,1)$ forms  in certain special Sobolev spaces that are strictly smaller than the standard ones and involve a loss of derivatives. In the special case of planar product domains,  we obtain Sobolev estimates assuming that each factor $\Omega_j$ to be merely Lipschitz, which extends the $L^p\to L^p$ estimate from \cite{LiLongLuo}.
Meanwhile, we show the existence of $\dbar$-solutions on space of distributions, a result that is  novel even for polydisks. Previously the similar solvability on distributions (with large orders) were only known for strongly pseudoconvex domains by Shi-Yao \cite{ShiYaoCk} and Yao \cite{YaoSPsiCXC2} and for convex domains of finite type by Yao \cite{YaoConvexFiniteType}. Moreover, we derive the optimal estimates for general $(0,q)$ forms with all $q\ge 1$. 
To the authors' best knowledge, for the case $q\ge 2$, the only previously  established  result for optimal $\dbar$ estimate on product domains is the  $L^2\to L^2$ estimate, which follows directly from H\"ormander's classical $L^2$-theory. In particular, the $L^p$-boundedness for $p\neq 2$ remains  open for the canonical solutions 
on $(0,q)$ forms even on polydisk. 

Our construction of the homotopy  operators is inspired from Nijenhuis-Woolf's formulae in \cite[(2.2.2)]{NijenhuisWoolf} for products of planar domains, see Theorem \ref{Thm::ProdHT} and Remark~\ref{Rmk::ActualNWFormula}. For estimates we use the so-called Fubini decomposition of Sobolev spaces, see Proposition \ref{Prop::Fubini}. 

In fact, the proof yields   homotopy formulae and the corresponding operator estimates   on a considerably larger class of product domains, provided that  each factor domain admits its  homotopy formulae and regularity estimates. To be precise, the following is the conditional result:

\begin{thm}\label{Thm::HTGeneral}
  Let $\Omega_j\subset\C^{n_j}$ be a bounded Lipschitz domain for each $j=1,\dots, m$, with   $m\ge 1$. Suppose  there exist linear homotopy operators $H^{\Omega_j}_q: C^\infty(\overline{\Omega_j};\wedge^{0,q})\to\Ds'(\Omega_j;\wedge^{0,q-1})$ for $1\le q\le n_j $, such that the following homotopy formulae hold (we set $H_{n_j+1}^{\Omega_j}=0$):
   \begin{equation}\label{Eqn::HTGeneral::htf1}
                 f= \dbar H_q^{\Omega_j} f + H_{q+1}^{\Omega_j}\dbar  f  \ \ \  \text{for all} \ \ \ f\in C^\infty (\overline{\Omega_j};  \wedge^{0,q}),\qquad1\le q\le n_j.
    \end{equation}
    Let $\Omega:=\Omega_1\times\dots\times\Omega_m$ and $n: = \sum_{j=1}^m n_j$. Then
    \begin{enumerate}[(i)]
        \item\label{Item::HTGeneral::HT} there exist linear operators $\Hc_q=\Hc_q^\Omega: C^\infty(\overline{\Omega};\wedge^{0,q})\to\Ds'(\Omega;\wedge^{0,q-1})$ for $1\le q\le n$, such that
            \begin{equation}\label{Eqn::HTGeneral::htfp}
                f=\dbar\Hc_qf+\Hc_{q+1}\dbar f \ \ \text{ for all}\ \  f\in C^\infty(\overline{\Omega};\wedge^{0,q}).
            \end{equation}
    \end{enumerate}
    
Further,  set the skew Bergman projections $P^{\Omega_j}: C^\infty(\overline{\Omega_j})\to\Ds'(\Omega_j)$ for $1\le j\le m$ by $P^{\Omega_j}f:=f-H_1^{\Omega_j}\dbar f$ for $f\in C^\infty(\overline{\Omega_j})$,  and $\Pc=\Pc^\Omega:C^\infty(\overline{\Omega})\to\Ds'(\Omega)$ by $\Pc f:=f-\Hc_1^\Omega \dbar f$ for $f\in C^\infty(\overline\Omega)$.
    \begin{enumerate}[(i)]\setcounter{enumi}{1}
        \item\label{Item::HTGeneral::SobBdd} Suppose there exists some $k\in\Z$ and $1<p<\infty$, such that for $1\le j\le m$ and $1\le q\le n_j$, $P^{\Omega_j}$ and $H^{\Omega_j}_q$ are both defined and bounded
    \begin{equation}\label{Eqn::HTGeneral::est}
          P^{\Omega_j}: W^{l,p}(\Omega_j)\to W^{l,p}(\Omega_j); \quad H_q^{\Omega_j}:W^{l,p}(\Omega_j;\wedge^{0,q})\to W^{l,p}(\Omega_j;\wedge^{0,q-1})\qquad \text{for }l=0,k.
    \end{equation}
    Then          
         $\Pc$ and $\Hc_q$ ($1\le q\le n$) obtained in \ref{Item::HTGeneral::HT} admit Sobolev estimates $\Pc:W^{k,p}(\Omega)\to W^{k,p}(\Omega)$ and $\Hc_q:W^{k,p}(\Omega;\wedge^{0,q})\to W^{k,p}(\Omega;\wedge^{0,q-1})$ as well.
    \end{enumerate}
\end{thm}

For the precise formulation of $\Hc^\Omega$ using $(H^{\Omega_j})_{j=1}^m$, see Remarks~\ref{Rmk::ExactHqforProd} and \ref{Rmk::ActualNWFormula}.


\section{Sobolev Spaces and Fubini Property}

In this section, we give the precise definition for the function space $ W^{k,p}$, and discuss a Fubini property for Sobolev  norms on product domains. 

\begin{note}\label{Note::Space::Dist}  Let  $\Ss'(\R^N)$ be the space of tempered distributions. For a bounded open subset $U\subset\R^N$, we denote by $\Ds'(U)$ the space of distributions in $U$, by $\Ss'(U)=\{\tilde f|_U:\tilde f\in \Ss'(\R^N)\}$  the space of extendable distributions in $U$, and by $\Es'(U)$ the space of distributions with compact supports in $U$.
\end{note}
See \cite[(3.1) and Proposition~3.1]{RychkovExtension} for an equivalent description of $\Ss'(U)$. See also \cite[Lemma~A.13~(ii)]{YaoSPsiCXC2}.

\begin{defn}\label{Defn::Space::Sob}
    Let $U\subseteq\R^N$ be an open subset, $k\in\Z_{\ge0}$, and  $1\le p\le\infty$.   $W^{k,p}(U)$ is the standard Sobolev space with norm $$\|f\|_{W^{k,p}(U)}=\Big(\sum_{|\alpha|\le k}\|D^\alpha f\|_{L^p(U)}^p\Big)^{1/p}.$$
    We denote by $W^{-k,p}(U):=\{\sum_{|\alpha|\le k}D^\alpha g_\alpha:g_\alpha\in L^p(U)\}$, a subset of distributions, with norm
    \begin{equation}\label{Eqn::DefnWkp::k-Norm}
        \|f\|_{W^{-k,p}(U)}=\inf\Big\{\Big(\sum_{|\alpha|\le k}\|g_\alpha\|_{L^p(U)}^p\Big)^{1/p}:f=\sum_{|\alpha|\le k}D^\alpha g_\alpha\text{ as distributions}\Big\}.
    \end{equation}
    Here when $p=\infty$ we take the usual modification by replacing the $\ell^p$ sum by the supremum.
    
    For $l\in\Z$, we use $W^{l,p}_c(U)\subset W^{l,p}(U)$ to be the subspace of all functions in $  W^{l,p}(U)$ that have compact supports in $U$.
\end{defn}

\begin{rem}\label{Rmk::RmkWkp}
    \begin{enumerate}[(i)]
        \item\label{Item::RmkWkp::Dual} For $k\ge0$ and $1\le p<\infty$, let $W_0^{k,p}(U)\subseteq W^{k,p}(U)$ be the closure of $C_c^\infty(U)$ in $\|\cdot\|_{W^{k,p}(U)}$. Then we have correspondence $W^{-k,p}(U)=W^{k,p'}_0(U)'$ with equivalent norms, where $U\subset\R^N$ is an arbitrary domain and $p'=p/(p-1)$. See e.g. \cite[Theorem~3.12]{AdamsSobolevSpaces}. It is also worth noticing that $W^{-k,p}(\R^N)=W^{k,p'}(\R^N)'$ via e.g. \cite[Corollary~3.23]{AdamsSobolevSpaces}.
        \item\label{Item::RmkWkp::Limit} When $U$ is a bounded Lipschitz domain we have $\Ss'(U)=\bigcup_{k=1}^\infty W^{-k,p}(U)$, see e.g. \cite[Lemma~A.13~(ii)]{YaoSPsiCXC2}. In other words if we have an operator $T$ defined on $W^{-k,p}(U)$ for all $k\ge1$, then $T$ is defined on $\Ss'(U)$.
    \end{enumerate}
\end{rem}

In order to incorporate the proof of Propositions~\ref{Prop::HTPlanar}, \ref{Prop::HTCXFinite} and \ref{Prop::WeiEst} we include the discussion of the fractional Sobolev spaces as well.
\begin{defn}[Sobolev-Bessel]\label{Defn::Hsp}
   Let $s\in\R$ and  $1<p<\infty$. We define the Bessel potential space $H^{s,p}(\R^N)$ to be the set of all tempered distributions $f\in\Ss'(\R^N)$ such that
    \begin{equation*}
        \|f\|_{H^{s,p}(\R^N)}:=\|(I-\Delta)^\frac s2f\|_{L^p(\R^N)}<\infty.
    \end{equation*}
    On an open subset $U\subseteq\R^N$,   define $H^{s,p}(U):=\{\tilde f|_U:\tilde f\in H^{s,p}(\R^N)\}$,   with norm $$\|f\|_{H^{s,p}(U)}:=\inf\big\{\|\tilde f\|_{H^{s,p}(\R^N)}: \tilde f\in H^{s,p}(\R^N),\  \tilde f|_U=f\big\}.$$
    We also define $\widetilde H^{s,p}(\overline U):=\{f\in H^{s,p}(\R^N):f|_{\overline{U}^c}=0\}$ as a closed subspace of $H^{s,p}(\R^N)$.
\end{defn}
    Here we use the standard (negative) Laplacian $\Delta=\sum_{j=1}^ND_{x_j}^2$. The fractional Laplacian (Bessel potential) can be defined via Fourier transform $((I-\Delta)^{s/2}f)^\wedge(\xi)=(1+4\pi^2|\xi|^2)^{s/2}\hat f(\xi)$, where $\hat f(\xi)=\int_{\mathbb R^N} f(x)e^{-2\pi ix\cdot\xi}d\xi$.

\begin{rem}\label{Rmk::RmkHsp}
Let $U\subseteq\R^N$ be a bounded Lipschitz domain and $1<p<\infty$. 

\begin{enumerate}[(i)]
    \item \label{Item::RmkHsp::Wkp=Hkp}    $H^{k,p}(U)=W^{k,p}(U)$ for all $k\in\Z$ with equivalent norms. See e.g. \cite[Lemma~A.11]{YaoSPsiCXC2} for a proof.
    \item\label{Item::RmkHsp::Interpo} The complex interpolation $[H^{s_0,p}(U),H^{s_1,p}(U)]_{\theta}=H^{(1-\theta)s_0+\theta s_1,p}(U)$ holds for all $s_0,s_1\in\R$ and $0<\theta<1$. See e.g. \cite[(1.372)]{TriebelTheoryOfFunctionSpacesIII}, where in the reference $H^{s,p}(U)=\Fs_{p2}^s(U)$ are special case of  Triebel-Lizorkin spaces. As a result our homotopy operators in Theorem~\ref{Thm::DbarThm} are in fact bounded $H^{s,p}(U)\to H^{s,p}(U)$ for all $s\in\R$ and $1<p<\infty$.
\end{enumerate}

\end{rem}

\begin{lem}\label{Lem::Extension}
Let $U\subset\R^N$ be a bounded Lipschitz domain. There is an extension operator $E:\Ss'(U)\to\Es'(\R^N)$ such that $E:W^{k,p}(U)\to W^{k,p}_c(\R^N)$ is bounded for all $k\in\Z$ and $1<p<\infty$. In particular for each $k$ and $1<p<\infty$ we have $W^{k,p}(U)=\{\tilde f|_U:\tilde f\in W^{k,p}(\R^N)\}$. 
\end{lem}
\begin{proof}
    An existence of such extension operator $E$ is established by Rychkov \cite[Theorem~4.1]{RychkovExtension}. In the reference we use the fact that $W^{k,p}(U)=\Fs_{p2}^k(U)$ are Triebel-Lizorkin spaces (see e.g. \cite[Lemma~A.11~(ii)]{YaoSPsiCXC2}). 
\end{proof}

Let us recall the \textit{Fubini property} for Sobolev norms on product domains. This will  enable a more convenient derivation of the Sobolev estimates for the homotopy operators.  Throughout the rest of the paper, we say two quantities $a$ and $b$ to satisfy $a\lesssim b$ if there exists some   constant $C$   such that $a\le Cb$. We say $a\approx b$ if $a\lesssim b$ and $b\lesssim a$ at the same time.  

\begin{prop}[Fubini Property]\label{Prop::Fubini}
    Let $U\subset\R^m$ and $V\subset\R^n$ be two bounded Lipschitz domains. Let $k\in\Z_+$ and $1<p<\infty$. Then
    \begin{enumerate}[(i)]
        \item\label{Item::Fubini::k+} $W^{k,p}(U\times V)=L^p(U;W^{k,p}(V))\cap L^p(V;W^{k,p}(U))$ in the sense that we have equivalent norms 
        $$\|f\|_{W^{k,p}(U\times V)}^p\approx_{U,V,k,p}\int_V\|f(\cdot,v)\|_{W^{k,p}(U)}^pdv+\int_U\|f(u,\cdot)\|_{W^{k,p}(V)}^pdu,$$
        provided either side is finite.
        \item\label{Item::Fubini::k-} $W^{-k,p}(U\times V)=L^p(U;W^{-k,p}(V))+ L^p(V;W^{-k,p}(U))$ in the sense that we have equivalent norms 
        $$\|f\|_{W^{-k,p}(U\times V)}^p\approx_{U,V,k,p}\inf_{f_1+f_2=f}\int_U\|f_1(u,\cdot)\|_{W^{-k,p}(V)}^pdu + \int_V\|f_2(\cdot,v)\|_{W^{-k,p}(U)}^pdv,$$
         provided either side is finite.
     
    \end{enumerate}
\end{prop}
\begin{rem}
    Here $L^p(U;X)$ can be interpreted as the space of strongly measurable functions which take values in a Banach space $X$, see e.g. \cite[Section~1.2.b]{AnalysisinBanachI} for more discussion.
\end{rem}

Essentially, Proposition  \ref{Prop::Fubini} shows that   for the Sobolev functions on the product of two domains, the $L^p$ norm of the  mixed derivatives across the two domains  can be controlled by the $L^p$ norm of the pure derivatives in each individual domain.

\begin{proof}
    When $U$ and $V$ are the total spaces $\mathbb R_u^m$ and $\mathbb R_v^n$, respectively,  the decomposition 
    \begin{equation}\label{Eqn::Fubini::k+Rn}
        W^{k,p}(\R^m_u\times \R^n_v)=L^p(\R^m_u;W^{k,p}(\R^n_v))\cap L^p(\R^n_v;W^{k,p}(\R^m_u)),\qquad k\ge1,\quad 1<p<\infty
    \end{equation}
    is a standard result, see e.g. \cite[Chapter~2.5.13]{TriebelTheoryOfFunctionSpacesI}. Recall from Remark~\ref{Rmk::RmkWkp}~\ref{Item::RmkWkp::Dual} $W^{-k,p}(\R^r)=W^{k,p'}(\R^r)'$ for $1<p<\infty$, $\frac{1}{p}+\frac{1}{p'} =1$ and $r\in\{m,n,m+n\}$, taking duality (see also \cite[Proposition~1.3.3]{AnalysisinBanachI}) we get 
    \begin{equation}\label{Eqn::Fubini::k-Rn}
        W^{-k,p}(\R^m_u\times \R^n_v)=L^p(\R^m_u;W^{-k,p}(\R^n_v))+ L^p(\R^n_v;W^{-k,p}(\R^m_u))
    \end{equation}
    with equivalent norms.
    
    \smallskip\noindent\ref{Item::Fubini::k+}: Now $k\ge0$. Clearly $ W^{k,p} (U\times V)\subset L^p(U;W^{k,p}(V))$ and $W^{k,p} (U\times V)\subset L^p(V;W^{k,p}(U))$, both of which are continuous embeddings. This gives $W^{k,p} (U\times V)\subseteq L^p(U;W^{k,p}(V))\cap L^p(V;W^{k,p}(U))$.
    
    Conversely, let $E^U:\Ss'(U)\to\Ss'(\R^m)$ and $E^V:\Ss'(V)\to \Ss'(\R^n)$ be the extension operators given in Lemma~\ref{Lem::Extension}. Define $\Ec^{U\times V}:=E^U\otimes E^V$  such that for $f(u,v)=g(u)h(v)$ we have $(\Ec^{U\times V}f)(u,v)=(E^Ug)(u)(E^Vh)(v)$. Clearly $$\Ec^{U\times V}=(E^U\otimes\id^{\R^n})\circ(\id^U\otimes E^V)=(\id^{\R^m}\otimes E^V)\circ(E^U\otimes\id^V),$$ where $(E^U\otimes\id^V)f(u,v)=E^U(f(\cdot,v))(u)$ for $(u,v)\in\R^m\times V$ and similarly for the rest. Therefore we have the following boundedness
    \begin{equation*}
        \Ec^{U\times V}:L^p(U;W^{k,p}(V))\xrightarrow{\id^U\otimes E^V}L^p(U;W^{k,p}(\R^n))\xrightarrow{E^U\otimes\id^{\R^n}} L^p(\R^m;W^{k,p}(\R^n)).
    \end{equation*}
    Similarly $\Ec^{U\times V}:L^p(V;W^{k,p}(U))\to L^p(\R^n;W^{k,p}(\R^m))$ as well. 
    
    Therefore by \eqref{Eqn::Fubini::k+Rn}, for every $f\in W^{k,p}(U\times V)$,
    \begin{align*}
        \|f\|_{W^{k,p}(U\times V)}\le&\|\Ec^{U\times V}f\|_{W^{k,p}(\R^m\times\R^n)}\approx \|\Ec^{U\times V}f\|_{L^p(\R^m;W^{k,p}(\R^n))}+\|\Ec^{U\times V}f\|_{L^p(\R^n;W^{k,p}(\R^m))}
        \\\lesssim&\|f\|_{L^p(U;W^{k,p}(V))}+\|f\|_{L^p(V;W^{k,p}(U))}.
    \end{align*}
    We conclude that \ref{Item::Fubini::k+} holds.

    \smallskip\noindent\ref{Item::Fubini::k-}: Clearly $L^p(U;W^{-k,p}(V))\subset W^{-k,p}(U\times V)$ and $ L^p(V;W^{-k,p}(U))\subset W^{-k,p}(U\times V)$ by \eqref{Eqn::DefnWkp::k-Norm}. This gives the embedding $L^p(U;W^{-k,p}(V))+L^p(V;W^{-k,p}(U))\subseteq W^{-k,p}(U\times V)$.
    
    Conversely, for every $f\in W^{-k,p}(U\times V)$, by Lemma~\ref{Lem::Extension} it admits an extension $\tilde f\in W^{-k,p}(\R^m\times\R^n)$. By \eqref{Eqn::Fubini::k-Rn} there exist  $\tilde f_1\in L^p(\R^m_u;W^{-k,p}(\R^n_v))$ and $\tilde f_2\in L^p(\R^n_v;W^{-k,p}(\R^m_u))$ such that $\tilde f=\tilde f_1+\tilde f_2$. Taking restrictions we get the existence of the decomposition $f=f_1+f_2$ where $f_1: =\tilde f_1|_{U\times V}\in L^p(U;W^{-k,p}(V))$ and $f_2: =\tilde f_2\in L^p(V;W^{-k,p}(U))$.
    
    Now for given $f\in W^{-k,p}(U\times V)$, let $f_1\in L^p(U;W^{-k,p}(V)) $ and $f_2\in L^p(V;W^{-k,p}(U))$ be arbitrary functions such that $f_1+f_2=f$, which exist from above. Thus $$\|f\|_{W^{-k,p}(U\times V)}= \|f_1\|_{W^{-k,p}(U\times V)} +\|f_2\|_{W^{-k,p}(U\times V)} \lesssim \|f_1\|_{L^p(U;W^{-k,p}(V))} +\|f_2\|_{L^p(V;W^{-k,p}(U))} 
    .    $$ Therefore by taking infimum over the decomposition $f_1+f_2=f$ we conclude that
    $$ \|f\|_{W^{-k,p}(U\times V)}^p\lesssim \inf_{f_1+f_2=f}\int_U\|f_1(u,\cdot)\|_{W^{-k,p}(V)}^pdu + \int_V\|f_2(\cdot,v)\|_{W^{-k,p}(U)}^pdv.$$
    That is,  $W^{-k,p}(U\times V)\subseteq L^p(U;W^{-k,p}(V))+L^p(V;W^{-k,p}(U))$, completing the proof of \ref{Item::Fubini::k-}.
\end{proof}

\begin{cor}\label{Cor::LiftBdd}
    Let $U\subset\R^m,V\subset\R^n$ be two bounded Lipschitz domains, and  $k\in\Z$, $1<p<\infty$.    
    Let $T:L^p(U)\to L^p(U)$ be a bounded linear operator that extends to a bounded linear map $T:W^{k,p}(U)\to W^{k,p}(U)$.     
    Defines $\Tc:L^p(U\times V)\to L^p(U\times V)$ by acting $T$ on the coordinate component of $U$, i.e. $\Tc f(u,v):=T(f(\cdot,v))(u)$. Then $\Tc:W^{k,p}(U\times V)\to W^{k,p}(U\times V)$ is defined and bounded.
\end{cor}
Here we are indeed using $\Tc=T\otimes\id^V$. See also Convention~\ref{Conv::OpProd}.
\begin{proof}

    First by definition $\Tc D_v=D_v\Tc$,  which ensures that $\Tc:L^p(U;W^{k,p}(V))\to L^p(U;W^{k,p}(V))$. The boundedness $\Tc:L^p(V;W^{k,p}(U))\to L^p(V;W^{k,p}(U))$ is a direct consequence of that of $T: W^{k,p}(U)\to W^{k,p}(U)$. The $W^{k,p}(U\times V)\to W^{k,p}(U\times V)$ bound of $\Tc$  then follows from Proposition~\ref{Prop::Fubini}.
\end{proof}

\section{Nijenhuis-Woolf Formulae and the proof of Theorem \ref{Thm::HTGeneral}}
In this section, we shall construct  homotopy formulae on product domains making use of an idea in \cite{NijenhuisWoolf}. This together with Proposition \ref{Prop::Fubini} allows us to prove Theorem \ref{Thm::HTGeneral}. 

First we introduce some notations and conventions for linear operators defined on forms (of mixed degrees), which will be used   to facilitate our proof.  

\begin{conv}[Spaces on forms]\label{Cov:;spacef}
    Let $\Xs\in\{\Ss',\Ds',C^\infty,W^{k,p},H^{s,p}:k\in\Z,s\in\R,1<p<\infty\}$ and let $U\subseteq\C^n$. For $1\le q\le n$ we use $\Xs(U;\wedge^{0,q})$ the space of $(0,q)$ forms $f(\zeta)=\sum_{|I|=q}f_I(\zeta)d\bar\zeta^I$ where $f_I\in\Xs(U)$ for all $I$. If $\Xs\in\{W^{k,p},H^{s,p}\}$, then we use the norm $\|f\|_{\Xs(U;\wedge^{0,q})}=\sum_{|I|=q}\|f_I\|_{\Xs(U)}$. Denote by $\Xs(U;\wedge^{0,\bullet})=\bigoplus_{q=0}^n\Xs(U;\wedge^{0,q})$ for forms of mixed degrees.
      \end{conv}

We adopt the following convention to extend an operator originally defined on forms of a single degree to one on forms of mixed degrees. 

\begin{conv}[Operators on mixed degree forms] \label{Conv::MixForm}
    Let $U\subseteq\C^n$, $0\le q,r\le n$ and $\Xs,\Ys\in\{\Ss',\Ds',C^\infty,W^{k,p},H^{s,p}\}$.     
    We identify a linear operator $S:\Xs(U;\wedge^{0,q})\to\Ys(U;\wedge^{0,r})$ as $S:\Xs(U;\wedge^{0,\bullet})\to\Ys(U;\wedge^{0,\bullet})$ by setting $S(f_Jd\bar \zeta_J)=0$ if $|J|\neq q$. 

    For a family of operators $(T_q)_{q=0}^n$ where each $T_q$ is defined on $(0,q)$ forms, we   use $T  =\sum_{q=0}^nT_q$ to denote the  corresponding  operator on mixed degree forms. Namely, for a form $f(z)=\sum_{q=0}^nf_q(z)$ where $f_q(z)$ is of degree $(0,q)$, we define $Tf=\sum_{q=0}^nT_qf_q$.    
\end{conv}
\begin{rem}\label{Rmk::MixForm}
    Under this convention, we can rewrite the homotopy formulae in Theorems~\ref{Thm::DbarThm} and \ref{Thm::HTGeneral} as a single formula (here we use $\Hc_0=\Hc_{n+1}=0$)
    $$f=\Pc f+\dbar\Hc f+\Hc\dbar f,\quad\text{for}\quad f\in\Ss'(\Omega;\wedge^{0,\bullet}),\qquad\text{where}\quad\Hc=\textstyle\sum_{q\ge1}\Hc_q.$$
\end{rem}

\begin{rem}
    For an operator $S$  defined on functions, namely, with $q=r=0$ in  Convention \ref{Conv::MixForm}, one  extends $S$ on differential forms by taking zero value on $(0,q)$ forms when $q\ge1$. In the paper not all the operators on functions follow this convention. For example for an extension operator $E:\Xs(U)\to\Xs(\C^n)$ we define it on forms by acting on components, i.e. $(Ef)(z)=\sum_I(Ef_I)(z)d\bar z^I$.
\end{rem}


Next, we extend operators originally defined on slices to the entire product domain using the following convention. 

\begin{conv}[Operator on product domains]\label{Conv::OpProd}
    Let $U\subseteq\C^m$ and $V\subseteq\C^n$ be two open sets, endowed with standard coordinate system $z=(z^1,\dots,z^m)$ and $\zeta=(\zeta^1,\dots,\zeta^n)$ respectively. For a linear operator $T^U:\Xs(U;\wedge^{0,\bullet})\to\Ys(U;\wedge^{0,\bullet})$, we denote $\Tc^U$ for the associated operator $T^U\otimes\id^V$ on $(0,\bullet)$ forms defined on $U\times V$ by setting
    \begin{equation}\label{Eqn::OpProd}
    \Tc^U(\omega\wedge d\bar \zeta^K)(z,\zeta):=T^U(\omega(\cdot,\zeta))(z)\wedge d\bar\zeta^K,\quad\text{where }\omega(z,\zeta)=\sum_I\omega_I(z,\zeta)d\bar z^I.
    \end{equation}
    
In particular, if we write $T^U(\sum_Jg_Jd\bar z^J)=:\sum_{I,J}(T^U_{IJ}g_I)d\bar z^J$ where $T^U_{IJ}:\Xs(U)\to\Ys(U)$ are linear operators on functions, then for a form $f(z,\zeta)=\sum_{J,K}f_{JK}(z,\zeta)d\bar z^J\wedge d\bar \zeta^K$,
    \begin{equation*}
        (\Tc^Uf)(z,\zeta)=\sum_{I,J,K}\big\{T^U_{IJ}[f_{IK}(\cdot,\zeta)]\big\}(z)d\bar z^I\wedge d\bar \zeta^K,\quad z\in U,\quad\zeta\in V.
    \end{equation*}
\end{conv}

Motivated by a one-dimensional analogue in \cite[(2.2.2) - (2.2.5)]{NijenhuisWoolf}, we deduce our homotopy formulae  making use of  the following product-type configuration. See also Remark~\ref{Rmk::ExactHqforProd}.

\begin{thm}[Product homotopy formulae]\label{Thm::ProdHT}
    Let $U\subset\C^{n_U}$ and $V\subset\C^{n_V}$ be two open subsets. Suppose  for each $W\in \{ U, V\}$, there exist continuous linear operators $H^{W}_q:C^\infty(\overline W;\wedge^{0,q})\to \Ds'(W;\wedge^{0,q-1})$ for $1\le q\le n_W$, such that the following homotopy formulae  hold ($H^W_{n_W+1}=0$ as usual) for $1\le q\le n_W$:
   \begin{equation}\label{htf}
                  f= \dbar H_q^{W} f + H_{q+1}^{W}\dbar  f\quad\text{for all}\quad f\in C^\infty (\overline W;  \wedge^{0,q}).
          \end{equation}
    Set $P^Wf:=f-H_1^W\dbar f$ for functions $f\in C^\infty(\overline W)$.
    \begin{enumerate}[(i)]
        \item\label{Item::ProdHT::HT} Then we have homotopy formulae $f=\Pc^{U\times V} f+\dbar\Hc^{U\times V} f+\Hc^{U\times V}\dbar f$ for $f\in C^\infty(\overline{U\times V};\wedge^{0,\bullet})$, where (see Conventions~\ref{Conv::MixForm} and \ref{Conv::OpProd}) 
        \begin{gather}\label{Eqn::ProdHT::DefPH}
            \Pc^{U\times V}:=\Pc^U\circ\Pc^V=P^U\otimes P^V;\qquad \Hc^{U\times V}:=\Hc^U+\Pc^U\circ\Hc^V=H^U\otimes\id^V+P^U\otimes H^V.
        \end{gather}
        \item\label{Item::ProdHT::Bdd} Let $k\in\Z$ and $1<p<\infty$. Suppose further that $U$ and $V$ are bounded Lipschitz domains, $P^U,H^U:W^{l,p}(U;\wedge^{0,\bullet})\to W^{l,p}(U;\wedge^{0,\bullet})$ and $H^V:W^{l,p}(V;\wedge^{0,\bullet})\to W^{l,p}(V;\wedge^{0,\bullet})$ are all defined and bounded for $l\in\{0,k\}$, then $\Hc^{U\times V}$ given in \eqref{Eqn::ProdHT::DefPH} are defined and bounded in $W^{k,p}(U\times V;\wedge^{0,\bullet})\to W^{k,p}(U\times V;\wedge^{0,\bullet})$ as well.

        If in addition $P^V:W^{l,p}(V)\to W^{l,p}(V)$ is bounded for $l\in\{0,k\}$, then $\Pc^{U\times V}:W^{k,p}(U\times V)\to W^{k,p}(U\times V)$ is bounded as well.
    \end{enumerate}
\end{thm}
\begin{rem}
    It is worth pointing out that the original formulae of Nijenhuis-Woolf in \cite{NijenhuisWoolf} are restricted to products of planar domains. Under their settings the $\Hc^{\Omega_j}$  operators (denoted by $T^j$ in the reference) are the solid Cauchy integral over $\Omega_j$. In their notation the operators are defined on functions rather than $(0,1)$ forms. Moreover, even in the case when each $\Omega_j$ is smooth, while $\Hc^{\Omega_j}$ there satisfies the desired $W^{k, p}$ regularity   for $k\ge 0$ (see, for instance, \cite[Proposition 3.1]{PanZhang} with $\mu\equiv 1$ there),  their $ \Pc^{\Omega_j}$ operators (denoted by $S^j$ in the reference), which are given by the boundary Cauchy integrals over $b \Omega_j$,   do not yield well-defined or bounded mappings on the $L^p$  space  as required in Theorem \ref{Thm::ProdHT}~\ref{Item::ProdHT::Bdd}. In Proposition \ref{Prop::HTPlanar} in the next section, we shall introduce a slightly different choice of homotopy operators to overcome this issue.
\end{rem}

\begin{rem}[Formulae with separated degrees]\label{Rmk::ExactHqforProd}
    
    Let $z=(z^1,\dots,z^{n_U})$ and $\zeta=(\zeta^1,\dots,\zeta^{n_V})$ be standard coordinate systems for $\C^{n_U}$ and $\C^{n_V}$, respectively.
    For $0\le j\le n_U$ and $0\le k\le n_V$, let us define the standard projection $\pi_{j,k}$ of forms by
    \begin{equation*}
        \pi_{j,k}f:=\sum_{|J|=j,|K|=k}f_{JK}d\bar z^J\wedge d\bar \zeta^K,  \quad \text{for every }f=\sum_{J\subseteq\{1,\dots,n_U\},K\subseteq\{1,\dots,n_V\}}f_{JK}d\bar z^J\wedge d\bar \zeta^K.
    \end{equation*}
    We have $f=\sum_{j=0}^{n_U}\sum_{k=0}^{n_V}\pi_{j,k}f$ and $\sum_{j+k=q}\pi_{j,k}f$ is the degree $(0,q)$ components of $f$.

    Under this notation,  and Conventions \ref{Conv::MixForm} and \ref{Conv::OpProd}, we can write $\Hc^{U\times V}_q$ in \eqref{Eqn::ProdHT::DefPH} for $1\le q\le n_U+n_V$ as
    \begin{equation*}
    \begin{split}
                \Hc^{U\times V}_q&= \Hc^U\circ \Big(\sum_{j+k=q }\pi_{j,k}\Big)+\Pc^U\circ\Hc^V\circ \Big(\sum_{j+k=q }\pi_{j,k}\Big)\\
                &=\sum_{j=1}^q\Hc^U_j\circ\pi_{j,q-j}+\Pc^U\circ\Hc^V_q\circ\pi_{0,q}=\sum_{j=1}^q(H^U_j\otimes\id^V)\circ\pi_{j,q-j}+(P^U\otimes H^V_q)\circ\pi_{0,q}.
        \end{split}
    \end{equation*}
Note that the formula \eqref{Eqn::ProdHT::DefPH} is asymmetric. Namely, if we swap $U$ and $V$, the homotopy operators in \eqref{Eqn::ProdHT::DefPH} are not the same.
  
\end{rem}
\begin{rem}
    It is important that $U$ and $V$ in the assumption are at most Lipschitz. In the proof of Theorem~\ref{Thm::DbarThm} and \ref{Thm::HTGeneral} we use induction with $U=\Omega_1\times\dots\times\Omega_{m-1}$ and $V=\Omega_m$. However, even for two smooth domains, the boundary regularity of their product is merely Lipschitz.

    In contrast, Theorem~\ref{Thm::ProdHT}~\ref{Item::ProdHT::Bdd} remains true for non-Lipschitz domains $U$ and $V$ if the analogy of Proposition~\ref{Prop::Fubini} holds for such $U$ and $V$. A typical example is the Hartogs triangle, which is a non-Lipschitz but uniform domain due to \cite{BurchardFlynnLuShawHartogsTriangle}.
\end{rem}

To prove Theorem~\ref{Thm::ProdHT}, we let $z=(z^1,\dots,z^{n_U})$ and $\zeta=(\zeta^1,\dots,\zeta^{n_V})$ be the standard coordinate systems for $\C^{n_U}$ and $\C^{n_V}$, respectively. We denote by $\dbar_z=\sum_{j=1}^{n_U}\dbar z^j\wedge \Coorvec{\bar z^j}$ and $\dbar_\zeta=\sum_{k=1}^{n_V}\dbar \zeta^k\wedge \Coorvec{\bar \zeta^k}$  the $\dbar$-operators of the $z$-component and the $\zeta$-component, respective. Note that on the product domain $U\times V$ we have $\dbar=\dbar_{z,\zeta}: =\dbar_z+\dbar_\zeta$. 

The key computation is the following (cf. \cite[(2.1.1) - (2.1.3)]{NijenhuisWoolf}):
\begin{lem}\label{Lem::ProdHTClaim}
    Keeping the notations as above and as in Theorem~\ref{Thm::ProdHT}, on $U\times V$ we have \begin{equation}\label{Eqn::ProdHT::Claim}
        \dbar_{z,\zeta}\Pc^U=\Pc^U\dbar_{z,\zeta},\quad \dbar_{z,\zeta}\Pc^V=\Pc^V\dbar_{z,\zeta},\quad\dbar_\zeta\Hc^U=-\Hc^U\dbar_\zeta,\quad\dbar_z\Hc^V=-\Hc^V\dbar_z.
    \end{equation}
Moreover, 
    \begin{equation}\label{Eqn::ProdHT::Claim2}
        \dbar\Hc^U+\Hc^U\dbar=\id-\Pc^U,\qquad \dbar\Hc^V+\Hc^V\dbar=\id-\Pc^V, 
    \end{equation}
    provided that the operators on both sides of the equalities are defined.
\end{lem}
\begin{proof}
    Recall that $P^U=\id_0-H^U_1\dbar_z$ is a projection for functions on $U$ to holomorphic functions on $U$. Therefore $\dbar_zP^U=0$. Since $P^U$ vanishes on $(0,q)$ forms when $q\ge1$, we get $P^U\dbar_z=0$. Together by Convention~\ref{Conv::OpProd} we have $\dbar_z\Pc^U=\Pc^U\dbar_z=0$.     
    Since $\Pc^U$ only acts on $z$-variable, using \eqref{Eqn::OpProd} as well we get $\dbar_\zeta\Pc^U=\Pc^U\dbar_\zeta$. Together we have $\dbar_{z,\zeta}\Pc^U=\Pc^U\dbar_{z,\zeta}$. The same argument yields $\dbar_{z,\zeta}\Pc^V=\Pc^V\dbar_{z,\zeta}$.
    
    Next, for a form $f(z,\zeta)=f_{JK}(z,\zeta)d\bar z^J\wedge d\bar\zeta^K$, the $\Hc^U(f_{JK}d\bar z^J)$ is a $(0,|J|-1)$ form. By a direct computation and \eqref{Eqn::OpProd},
    \begin{align*}
        &\dbar_\zeta\Hc^Uf=\dbar_\zeta\Hc^U(f_{JK}d\bar z^J\wedge d\bar\zeta^K)=\sum_{k=1}^{n_V}d\bar\zeta^k\wedge\Coorvec{\bar\zeta^k}\Hc^U(f_{JK}d\bar z^J)\wedge d\bar\zeta^K
        \\
        =&(-1)^{|J|-1}\sum_{k=1}^{n_V}\Coorvec{\bar\zeta^k}\Hc^U(f_{JK}d\bar z^J)\wedge d\bar\zeta^k\wedge d\bar\zeta^K=(-1)^{|J|-1}\sum_{k=1}^{n_V}\Hc^U\Big(\frac{\partial f_{JK}}{\partial\bar\zeta^k}d\bar z^J\wedge d\bar\zeta^k\wedge d\bar\zeta^K\Big)
        \\
        =&(-1)^{|J|-1}(-1)^{|J|}\sum_{k=1}^{n_V}\Hc^U\Big(d\bar\zeta^k\wedge\frac{\partial f_{JK}}{\partial\bar\zeta^k}d\bar z^J \wedge d\bar\zeta^K\Big)=-\Hc^U\dbar_\zeta(f_{JK}d\bar z^J\wedge d\bar\zeta^K)=-\Hc^U\dbar_\zeta f.
    \end{align*}
    We get $\dbar_\zeta\Hc^U=-\Hc^U\dbar_\zeta$. By swapping $(z,U)$ and $(\zeta,V)$ we get $\dbar_z\Hc^V=\Hc^V\dbar_z$. This completes the proof of \eqref{Eqn::ProdHT::Claim}.

    Since by assumption $\id^U-P^U=\dbar_zH^U+H^U\dbar_z$ and $\id^V-P^V=\dbar_\zeta H^V+H^V\dbar_\zeta$. Combing them with \eqref{Eqn::ProdHT::Claim} and Convention~\ref{Conv::OpProd}, we have \eqref{Eqn::ProdHT::Claim2}.
\end{proof}
\begin{proof}[Proof of Theorem~\ref{Thm::ProdHT}]
First we note that $\Pc^{U\times V},\Hc^{U\times V}$ from \eqref{Eqn::ProdHT::DefPH} are always defined on $C^\infty(\overline{U\times V};\wedge^{0,\bullet})$. Indeed, for $W\in\{U,V\}$, a continuous linear operator $T^W:C^\infty(\overline W)\to\Ds'(W)$ can be lifted as a linear operator $\widetilde T^W:C_c^\infty(\C^{n_W})\to\Ds'(W)$ via a  continuous extension operator $C^\infty(\overline W)\to C_c^\infty(\C^{n_W})$. By the Schwartz Kernel Theorem (see e.g. \cite[Theorem~51.7]{TrevesBook}) $\widetilde T^U\otimes\widetilde T^V:C_c^\infty(\C^{n_U}\times\C^{n_V})\to\Ds'(U\times V)$ is continuous. Clearly $((\widetilde T^U\otimes\widetilde T^V)\tilde f)|_{U\times V}=(T^U\otimes T^V)(\tilde f|_{U\times V})$ for all $\tilde f\in C_c^\infty(\C^{n_U+n_V})$, we get the continuity $T^U\otimes T^V:C^\infty(\overline{U\times V})\to\Ds'(U\times V)$. Take $T^U\in\{P^U,H^U\}$ and $T^V\in\{P^V,H^V\}$ we get the definedness. 

    Using \eqref{Eqn::ProdHT::Claim2} for every $f\in C^\infty(\overline{U\times V};\wedge^{0,\bullet})$,
    \begin{align*}
        &\Pc^{U\times V}f+\dbar\Hc^{U\times V}f+\Hc^{U\times V}\dbar f=\Pc^U\Pc^Vf+\dbar(\Hc^U+\Pc^U\Hc^V)f+(\Hc^U+\Pc^U\Hc^V)\dbar f
        \\
        =&\Pc^U\Pc^Vf+(\dbar\Hc^U+\Hc^U\dbar)f+(\dbar\Pc^U\Hc^V+\Pc^U\Hc^V\dbar)f=\Pc^U\Pc^Vf +f-\Pc^Uf+\Pc^U(\dbar\Hc^V+\Hc^V\dbar)f
        \\
        =&\Pc^U\Pc^Vf+f-\Pc^Uf+\Pc^Uf-\Pc^U\Pc^Vf=f.
    \end{align*}
    This gives the proof of \ref{Item::ProdHT::HT}.

    For \ref{Item::ProdHT::Bdd}, by Corollary~\ref{Cor::LiftBdd} the boundedness assumption of  $P^U,H^U:W^{k,p}_z(U;\wedge^{0,\bullet})\to W^{k,p}_z(U;\wedge^{0,\bullet})$ and $P^U,H^U:L^p_z(U;\wedge^{0,\bullet})\to L^p_z(U;\wedge^{0,\bullet})$ implies the boundedness of $\Pc^U,\Hc^U:W^{k,p}_{z,\zeta}(U\times V;\wedge^{0,\bullet})\to W^{k,p}_{z,\zeta}(U\times V;\wedge^{0,\bullet})$. The same argument yields the boundedness $\Hc^V:W^{k,p}_{z,\zeta}(U\times V;\wedge^{0,\bullet})\to W^{k,p}_{z,\zeta}(U\times V;\wedge^{0,\bullet})$. 
    By \eqref{Eqn::ProdHT::DefPH} with compositions, we conclude that $\Hc^{U\times V}:W^{k,p}_{z,\zeta}(U\times V;\wedge^{0,\bullet})\to W^{k,p}_{z,\zeta}(U\times V;\wedge^{0,\bullet})$ is bounded.
    
    If we further assume $P^V$ is bounded in both $L^p_\zeta(V)\to L^p_\zeta(V)$ and $W^{k,p}_\zeta(V)\to W^{k,p}_\zeta(V)$, then by Corollary~\ref{Cor::LiftBdd} $\Pc^V:W^{k,p}_{z,\zeta}(U\times V)\to W^{k,p}_{z,\zeta}(U\times V)$ is bounded as well. Taking compositions with $\Pc^U$, we obtain the boundedness $\Pc^{U\times V}:W^{k,p}_{z,\zeta}(U\times V)\to W^{k,p}_{z,\zeta}(U\times V)$, completing the proof.
\end{proof}

\begin{proof}[Proof of Theorem \ref{Thm::HTGeneral}] Identifying $\Pc$ as the operator on forms of all degrees following Convention~\ref{Conv::MixForm} and Remark~\ref{Rmk::MixForm},  we can write the homotopy formulae as $f=\Pc f+\dbar \Hc f+\Hc\dbar f$ for mixed degree form $f$. 

The proof can be done by induction on $m$. The based case $m=1$ follows from the assumption \eqref{Eqn::HTGeneral::htf1}.     Suppose the case $m-1$ is obtained.
    For the case $m$, take $U:=\Omega_1\times\dots\times\Omega_{m-1}\subset\C^{n_1+\dots+n_{m-1}}$ and $V:=\Omega_m\subset\C^{n_m}$. Since the product of bounded Lipschitz domains is still bounded Lipschitz, we see that $U$ and $V$ are both bounded Lipschitz domains as well. By the induction hypothesis there are linear operators $H^U=\sum_{q=1}^{n_1+\dots+n_{m-1}}H^U_q$ on $C^\infty(\overline{U};\wedge^{0,\bullet})$ and  $H^V=\sum_{q=1}^{n_m}H^V_q$ on $C^\infty(\overline{V};\wedge^{0,\bullet})$ (in terms of Conventions~\ref{Conv::MixForm} and \ref{Conv::OpProd}), such that $g=P^Ug+\dbar H^Ug+H^U\dbar g$ for all $g\in C^\infty(\overline{U};\wedge^{0,\bullet})$, $h=P^Vh+\dbar H^Vh+H^V\dbar h$ for all $h\in C^\infty(\overline{V};\wedge^{0,\bullet})$, where $P^U=\id_0^U-H^U_1\dbar$ and $P^V=\id_0^V-H^V_1\dbar$ are skew Bergman projections on functions (in $U$ and $V$ respectively).    

    Applying Theorem~\ref{Thm::ProdHT} to such $P^U,H^U,P^V,H^U$ we obtain the desired operators $\Pc^\Omega$ and $\Hc^\Omega=(\Hc_q^\Omega)_{q=1}^{n_1+\dots+n_m}$ on $\Omega=U\times V$.  By Theorem~\ref{Thm::ProdHT}~\ref{Item::ProdHT::HT} $f=\Pc^\Omega f+\dbar\Hc^\Omega f+\Hc^\Omega\dbar f$ for all $f\in C^\infty(\overline\Omega;\wedge^{0,\bullet})$, which gives \eqref{Eqn::HTGeneral::htfp}.  

    \smallskip
    Suppose further \eqref{Eqn::HTGeneral::est} holds, that is, for some $k\in\Z$ and $1<p<\infty$, $P^U,P^V,H^U,H^V$ are all  bounded in $W^{k,p}\to W^{k,p}$ and $L^p\to L^p$.    
    By  Theorem~\ref{Thm::ProdHT}~\ref{Item::ProdHT::Bdd} the $W^{k,p}$ and $L^p$ boundeness for $P^U,H^U,P^V,H^U$ implies the $W^{k,p}$ boundedness for $\Pc^\Omega$ and $\Hc^\Omega$. \ref{Item::HTGeneral::SobBdd}  is thus proved. 
\end{proof}

\begin{rem}\label{Rmk::ActualNWFormula}
    By expanding the induction, the formulae we have for $\Omega=\Omega_1\times\dots\times\Omega_r$ are
    \begin{align*}
        \Pc^\Omega=&\Pc^{\Omega_1}\dots\Pc^{\Omega_m}=P^{\Omega_1}\otimes\dots\otimes P^{\Omega_m};
        \\
        \Hc^\Omega=&\Hc^{\Omega_1}+\Pc^{\Omega_1}\Hc^{\Omega_2}+\dots+\Pc^{\Omega_1}\dots\Pc^{\Omega_{m-1}}\Hc^{\Omega_m}
        \\
        &=H^{\Omega_1}\otimes\id^{\Omega_2\times\dots\times\Omega_m}+P^{\Omega_1}\otimes H^{\Omega_2}\otimes\id^{\Omega_3\times\dots\times\Omega_m}+\dots+ P^{\Omega_1}\otimes \dots\otimes P^{\Omega_{m-1}}\otimes H^{\Omega_m}.
    \end{align*}
    For a given degree $(0, q)$, the precise expression of $\Hc^\Omega_q$ follows from the same deduction to Remark~\ref{Rmk::ExactHqforProd}.
\end{rem}

\section{Proof of Theorem~\ref{Thm::DbarThm}}

In this section, we check that for each factor $\Omega_j$  under consideration   in Theorem~\ref{Thm::DbarThm}, there exist linear operators $(H^{\Omega_j}_q)_{q=1}^{n_j}$ and $P$ which satisfy the homotopy formulae  and has the desired boundedness in all Sobolev spaces.

\begin{prop}\label{Prop::HTPlanar}
    Let $\Omega\subset\C$ be a bounded Lipschitz domain. Then there is an operator $H_1:\Ss'(\Omega;\wedge^{0,1})\to\Ss'(\Omega)$ such that $\dbar H_1=\id$ and $H_1:W^{k,p}(\Omega;\wedge^{0,1})\to W^{k+1,p}(\Omega)$ is bounded for all $k\in\Z$ and $1<p<\infty$. In particular $P:=\id-H_1\dbar$ satisfies $P:W^{k,p}(\Omega)\to W^{k,p}(\Omega)$ for all $k\in\Z$ and $1<p<\infty$, and we have the homotopy formula $f=Pf+\dbar Hf+H\dbar f$ for $f\in \Ss'(\Omega, \wedge^{0,\bullet})$.
\end{prop}
Note that since there are no $(0,2)$ forms in $\C^1$, we have $H=H_1$. In particular, $f=Pf+H_1\dbar f$ for functions $f\in\Ss'(\Omega)$ and $f=\dbar H_1f$ for $(0,1)$ forms $f\in\Ss'(\Omega;\wedge^{0,1})$.
\begin{proof}
    Take a bounded open set $U\Supset\Omega$. By Lemma~\ref{Lem::Extension} there exists an extension operator $E:\Ss'(\Omega)\to \Es'(U)$ such that $E:W^{k,p}(\Omega)\to W^{k,p}_c(U)$ is bounded for all $k\in\Z$ and $1<p<\infty$. Take  $$H_1(gd\bar z):=(\tfrac1{\pi z}\ast Eg)|_\Omega.$$ Since $\frac1{\pi z}$ is the fundamental solution to $\Coorvec{\bar z}$, we get $\dbar H_1(gd\bar z)=gd\bar z$ for all $g\in\Ss'(\Omega)$. 
    The boundedness $H_1:W^{k,p}(\Omega;\wedge^{0,1})\to W^{k+1,p}(\Omega)$ is standard, from which one simultaneously  obtains the boundedness $P:W^{k,p}(\Omega)\to W^{k,p}(\Omega)$. We give a version of the proof here.

    Since $E:W^{k,p}(\Omega)\to W^{k,p}_c(U)$ is bounded, it suffices to show the boundedness $[g\mapsto\frac1{\pi z}\ast g]:W^{k,p}_c(U)\to W^{k+1,p}(\Omega)$. Since $U$ is bounded, say $U\subset B(0,R)$, we can take a $\chi\in C_c^\infty(\C)$ such that $\chi|_{B(0,2R)}\equiv1$, which allows $(\frac1{\pi z}\ast g)|_\Omega=((\chi\cdot\frac1{\pi z})\ast g)|_\Omega$. Thus the proposition is further reduced to showing  $[g\mapsto(\chi\cdot\frac1{\pi z})\ast g]:W^{k,p}(\C)\to W^{k+1,p}(\C)$ is bounded.
    
    Recalling that for the Fourier transform $\hat f(\xi,\eta)=\int_\C f(x+iy)e^{-2\pi i(x\xi+y\eta)}dxdy$, we see that $$m(\xi,\eta):=\big((I-\Delta)^{\frac12}(\chi\cdot\tfrac1{\pi z})\big)^\wedge(\xi,\eta)=\tfrac1{\pi i}\sqrt{1+4\pi^2(|\xi|^2+|\eta|^2)}\cdot\Big(\check\chi\ast\tfrac1{\xi+i\eta}\Big).$$
    This is a bounded smooth function in $\R^2_{\xi,\eta}$ such that $\sup_{\xi,\eta}\sqrt{\xi^2+\eta^2}|\nabla m(\xi,\eta)|<\infty$, which is in particular a H\"ormander-Mikhlin multiplier. By the H\"ormander-Mikhlin multiplier theorem (see e.g. \cite[Section~6.2.3]{GrafakosClassical}) $[g\mapsto (I-\Delta)^{\frac12}(\chi\cdot\frac1{\pi z})\ast g]:L^p(\C)\to L^p(\C)$ is bounded for all $1<p<\infty$.

    Using the Sobolev-Bessel spaces in Definition~\ref{Defn::Hsp} and the fact that $(I-\Delta)^\frac s2(\check m\ast g)=\check m\ast (I-\Delta)^\frac s2g$, we conclude that $[g\mapsto (\chi\cdot\frac1{\pi z})\ast g]:H^{s,p}(\C)\to H^{s+1,p}(\C)$ is bounded for all $s\in\R$ and $1<p<\infty$. The  $W^{k,p}$ boundedness follows from Remark~\ref{Rmk::RmkHsp}~\ref{Item::RmkHsp::Wkp=Hkp}.
\end{proof}

\begin{prop}\label{Prop::HTSPsiCX}
    Let $\Omega\subset\C^n$ be a bounded domain which is either $C^2$ strongly pseudoconvex or $C^{1,1}$ strongly $\C$-linearly convex. There are   linear operators $P:\Ss'(\Omega)\to\Ss'(\Omega)$ and $H_q:\Ss'(\Omega;\wedge^{0,q})\to\Ss'(\Omega;\wedge^{0,q-1})$ for $1\le q\le n$, such that $f=Pf+\dbar Hf+H\dbar f$ for all $f\in \Ss'(\Omega, \wedge^{0,\bullet})$, and $P,H:W^{k,p}(\Omega;\wedge^{0, \bullet })\to W^{k,p}(\Omega;\wedge^{0, \bullet })$ are bounded for all $k\in\Z$ and $1<p<\infty$.
\end{prop}
See  \cite[Theorem~1.1]{YaoSPsiCXC2}. In fact we have the boundedness $H_q:H^{s,p}(\Omega;\wedge^{0,q})\to H^{s+1/2,p}(\Omega;\wedge^{0,q-1})$ for $1\le q\le n-1$ and $H_n:H^{s,p}(\Omega;\wedge^{0,n})\to H^{s+1,p}(\Omega;\wedge^{0,n-1})$ for all $s\in \mathbb R$ and $1<p<\infty$.

\begin{prop}\label{Prop::HTCXFinite}
    Let $\Omega\subset\C^n$ be a smooth convex domain of finite type. There are linear operators   $P:\Ss'(\Omega)\to\Ss'(\Omega)$ and $H_q:\Ss'(\Omega;\wedge^{0,q})\to\Ss'(\Omega;\wedge^{0,q-1})$ for $1\le q\le n$, such that $f=Pf+\dbar Hf+H\dbar f$ for all $f\in \Ss'(\Omega, \wedge^{0,\bullet})$, and $P,H:W^{k,p}(\Omega;\wedge^{0, \bullet })\to W^{k,p}(\Omega;\wedge^{0,\bullet})$ are bounded for all $k\in\Z$ and $1<p<\infty$.
\end{prop}
The boundedness of $H_q$ was obtained in  \cite{YaoSPsiCXC2}. For the boundedness of $P=\id_0-H_1\dbar$, we postpone the proof to Theorem~\ref{Thm::PforCXFinite} in Section~\ref{Section::CXFinite}. A slightly more general version of this statement using Triebel-Lizorkin spaces can be found in  the arxiv version \cite[Appendix B]{YaoConvexFiniteType} with a similar argument as in the Appendix. 



Theorem~\ref{Thm::DbarThm} now follows directly from Theorem~\ref{Thm::HTGeneral} with Propositions~\ref{Prop::HTPlanar} - \ref{Prop::HTCXFinite}. We include the proof for completeness.

\begin{proof}[Proof of Theorem~\ref{Thm::DbarThm} and Corollary \ref{Cor::dbar}]Since on each $\Omega_j$ we have \eqref{Eqn::HTGeneral::htf1} and \eqref{Eqn::HTGeneral::est} for all  $k\in \mathbb Z$ and $1<p<\infty$ by Propositions~\ref{Prop::HTPlanar}
- \ref{Prop::HTCXFinite}, we obtain the linear operators  $\Pc$ and $\Hc$ as defined  in Theorem \ref{Thm::HTGeneral}, which satisfy  \eqref{Eqn::HTGeneral::htfp}, and are bounded  on $W^{k,p} $ for all $k\in \mathbb Z$ and $1<p<\infty$. Because $C^\infty(\overline\Omega)$ is dense in $W^{k,p}(\Omega)$ (see e.g. \cite[Lemma~A.14]{YaoSPsiCXC2} for $k\le0$), the homotopy formulae uniquely extends to all $f\in W^{k,p}(\Omega;\wedge^{0,\bullet})$ for $k\le0$ and $1<p<\infty$. By Remark~\ref{Rmk::RmkWkp}~\ref{Item::RmkWkp::Limit} again the homotopy formula \eqref{Eqn::HTGeneral::htfp}  holds for all $f\in \Ss'(\Omega;\wedge^{0,\bullet})$. This proves  Theorem~\ref{Thm::DbarThm}. Corollary \ref{Cor::dbar} is a direct consequence of Theorem~\ref{Thm::DbarThm}.
\end{proof}

\begin{rem}
    If one only focuses on  optimal $W^{k,p}$ estimates for $k\ge0$, we can also allow $\Omega_j$ in Theorem ~\ref{Thm::DbarThm} to be a smooth pseudoconvex domain  of finite type  in $\C^2$, or other pseudoconvex domains where the canonical solution operators $H_q:=\dbar^*N_q$ and the Bergman projection $P:=\id_0-\dbar^*N_1\dbar$ are bounded in $W^{k,p}$. See e.g. \cite[Corollaries~7.5 and 7.6]{ChangeNagelStein1992}.

    However if one further looks for $W^{k,p}$ estimates for small enough $k< 0$, the canonical solutions will not work. This is due to the ill-posedness of the $\dbar$-Neumann problem on space of distributions. See \cite[Lemma~A.32]{YaoSPsiCXC2}. 
\end{rem}

\begin{rem}[Near optimal H\"older estimates]\label{Rmk::Holder}
    If we use \eqref{Eqn::ProdHT::DefPH} for the H\"older spaces, then we have end point optimal H\"older estimates $\Hc^\Omega:C^{k,\alpha}(\Omega;\wedge^{0, \bullet })\to C^{k,\alpha-}(\Omega;\wedge^{0, \bullet })$ for all $k\ge0$ and $0<\alpha<1$.
    This can be done by Sobolev embeddings as follows.  

    Indeed, for every $\eps>0$ by taking $n/\eps<p<\infty$, we have  continuous  embeddings $C^{k,\alpha}(\Omega)\hookrightarrow H^{k+\alpha,p}(\Omega)\hookrightarrow C^{k,\alpha-\eps}(\Omega)$, see e.g. \cite[Remark~1.96 and Theorem~1.122]{TriebelTheoryOfFunctionSpacesIII}. From Remark~\ref{Rmk::RmkHsp}~\ref{Item::RmkHsp::Interpo} we obtain the boundedness $\Pc^\Omega,\Hc^\Omega:H^{k+\alpha,p}(\Omega;\wedge^{0, \bullet })\to H^{k+\alpha,p}(\Omega;\wedge^{0, \bullet })$. Thus $C^{k,\alpha}(\Omega;\wedge^{0, \bullet })\to C^{k,\alpha-\eps}(\Omega;\wedge^{0, \bullet })$ is bounded. Letting $\eps\to0^+$ we get the end point optimal H\"older bounds.

\end{rem}



\appendix
\section{Skew Bergman Projection on Convex Domains of Finite type}\label{Section::CXFinite}


In this section we briefly review the construction of homotopy formulae on convex domains of finite type from \cite{YaoConvexFiniteType} and complete the proof to  Proposition~\ref{Prop::HTCXFinite}. 
\begin{thm}\label{Thm::PforCXFinite}
    Let $\Omega\subset\C^n$ be a smooth convex domain of finite type. For the homotopy operators $\Hc_q:\Ss'(\Omega;\wedge^{0,q})\to\Ss'(\Omega;\wedge^{0,q-1})$ for $q=1,\dots,n$ given in \cite[Theorem~1.1]{YaoConvexFiniteType}, let $\Pc f:=f-\Hc_1\dbar f$ for $f\in\Ss'(\Omega)$. Then $\Pc:H^{s,p}(\Omega)\to H^{s,p}(\Omega)$ is bounded for all $s\in\R$ and $1<p<\infty$. In particular $\Pc: W^{k,p}(\Omega)\to W^{k,p}(\Omega)  $ is   bounded for $k\in\Z$ and $1<p<\infty$.
\end{thm}

Here for a convex domain we can use affine line type \cite[Definition~3.1]{YaoConvexFiniteType} to define the type condition. See e.g. \cite{McNealFiniteType,BoasStraubeFiniteType} for more discussions.

We briefly review the construction. Let $\varrho:\C^n\to\R$ be a defining function of $\Omega$, which is a smooth function such that $\nabla\varrho(z)\neq0$ for all $\zeta\in b\Omega$ and $\Omega=\{\zeta\in\C^n:\varrho(\zeta)\neq0\}$. We can assume that there is a $T_1>0$ such that for all $|t|<T_1$ the sublevel set $\Omega_t:=\{\varrho<t\}$ are all scaled copies of $\Omega$, which in particular have the same finite type as   $\Omega$. 

Denote $U_1=\{\zeta:-T_1<\varrho(\zeta)<T_1\}$. For each $\zeta\in U_1$, we have orthogonal decomposition of the $(0,1)$ cotangent space $\C^n=T^{*0,1}_\zeta\C^n=(\Span_{\mathbb C} \dbar\varrho(\zeta))\oplus T^{*0,1}_\zeta(b\Omega_{\varrho(\zeta)})$. This leads to an orthogonal decomposition $f=f^\top+f^\bot$ for $(0,q)$ forms $f(\zeta)=\sum_If_I(z)\dbar \zeta^I$ defined in $U_1$:
\begin{itemize}
    \item $f^\bot$ is in the ideal generated by $\dbar\varrho$, i.e. $\iota_{Z}f^\bot=0$ for every $(0,1)$-vector fields $Z=\sum_{j=1}^nZ_j\Coorvec{\bar \zeta_j}$ such that $Z\varrho=0$.
    \item $f^\top$ is a section of $\coprod_\zeta\wedge^q T^{*0,1}_\zeta(b\Omega_{\varrho(\zeta)})$, i.e. $\iota_{\Coorvec{\bar\varrho}}f^\top=0$, where $\Coorvec{\bar\varrho}=|\dbar\varrho|^{-2}\sum_{j=1}^n\frac{\partial\varrho}{\partial\zeta_j}\Coorvec{\bar \zeta_j}$.
\end{itemize}
See \cite[Definition~2.6 and Remark~2.8]{YaoConvexFiniteType} for details. For a bidegree form $K(z,\zeta)$ in variables $z$ and $\zeta$, we use $K^\top(z,\zeta)$ and $K^\bot(z,\zeta)$ for the projections with respect to $\zeta$-variable but not to $z$-variable.

For $\zeta\in U_1$ we also define the so-called \textit{$\eps$-minimal ellipsoid} (associated to $\varrho$):
\begin{equation}\label{Eqn::SkewBerg::DefPEps}
    P_\eps(\zeta)=\Big\{\zeta+\sum_{j=1}^na_jv_j: a_1,\dots,a_n\in\C,\ \sum_{j=1}^n\frac{|a_j|^2}{\tau_j(\zeta,\eps)^2}<1\Big\},
\end{equation} where $(v_1,\dots,v_n)$ is a unitary basis called \textit{$\eps$-minimal basis} at $\zeta$ and $\tau_1(\zeta,\eps)\le\dots\le\tau_n(\zeta,\eps)$ are the side lengths. See \cite[Definition~3.2]{YaoConvexFiniteType} and \cite[Definition~2.6]{HeferMultitype}. Roughly speaking $\tau_j(\zeta,\eps)$ is the minimum number such that there is a unit vector $v_j$ satisfying $v_j\bot\Span_\C(v_1,\dots,v_{j-1})$ and $\varrho(\zeta+\tau_j(\zeta,\eps)\cdot v_j)=\varrho(\zeta)+\eps$.
This was first constructed by Yu in \cite{YuMultitype}. 

Recall from \cite[Lemma~3.3 and Remark~3.4]{YaoConvexFiniteType} that the following estimates hold: there is a $C_0>1$ and $\eps_0>0$ such that
\begin{enumerate}[label=(\arabic*)]\setcounter{enumi}{\value{equation}}
    \item\label{Eqn::FlipP} For every $0<\eps<\eps_0$ and $P_\eps(\zeta)\subset U_1$, if $z\in P_\eps(\zeta)$ then $\zeta\in P_{C_0\eps}(z)$;
    \item\label{Eqn::TauOrder} $C_0^{-1}\eps\le\tau_1(\zeta,\eps)\le C_0\eps$ and $\tau_n(\zeta,\eps)\le C_0\eps^{1/m}$, where $m$ is the type of $\Omega$.\setcounter{equation}{\value{enumi}}
\end{enumerate}
See also \cite[Section~2]{HeferMultitype} for more details.
We shall need the following estimates:
\begin{prop}[{\cite[Lemma~3.9]{YaoConvexFiniteType}}]\label{Prop::BasisEst}
    Let $\Omega  $, $\varrho$, $P_\eps(\zeta)$, $\tau_j(\zeta, \eps)$ and $\eps_0$ be defined as above. There is a neighborhood $\Uc$ of $\overline\Omega$ and a smooth $(1,0)$ form $\widehat Q(z,\zeta)=\sum_{j=1}^nQ_j(z,\zeta)d\zeta_j$ defined for $z\in\Omega$ and $\zeta\in\Uc\backslash\Omega$, such that:
\begin{enumerate}[(i)]
    \item $\widehat Q$ is a Leray form, i.e. $\widehat Q$ is holomorphic in $z$, and $|\widehat Q(z,\zeta)|\neq0$ for all $z\in\Omega$ and $\zeta\in \Uc\backslash\Omega$.
    \item Denote $\widehat S(z,\zeta):=\sum_{l=1}^n\widehat Q_l(z,\zeta)(\zeta_l-z_l)$.
    For every $k\ge0$ there is a $C_k>0$ such that for every $0\le j\le n-1$, $0<\eps\le\eps_0$, $\zeta\in \Uc\backslash\overline\Omega$ and $z\in \Omega\cap P_\eps(\zeta)\backslash P_{\eps/2}(\zeta)$,
\begin{equation}\label{Eqn::FiniteCX::BasisEst}
    \bigg|D^k_{z,\zeta}\Big(\frac{\widehat Q\wedge(\dbar \widehat Q)^j}{\widehat S^{j+1}}\Big)^\top(z,\zeta)\bigg|\le \frac{C_k\eps^{-1-k}}{\prod_{l=2}^{j+1}\tau_l(\zeta,\eps)^2};\quad
    \bigg|D^k_{z,\zeta}\Big(\frac{\widehat Q\wedge(\dbar \widehat Q)^j}{\widehat S^{j+1}}\Big)^\bot(z,\zeta)\bigg|\le \frac{C_k\eps^{-2-k}\tau_{j+1}(\zeta,\eps)}{\prod_{l=2}^{j+1}\tau_l(\zeta,\eps)^2}.
\end{equation}
\end{enumerate}
    Here $D^k=\{\partial^\alpha_z\partial^\beta_\zeta\partial^\gamma_{\overline\zeta}\}_{|\alpha+\beta+\gamma|\le k}$ is the collection of differential operators acting on the components of the forms.
\end{prop}
Here in the reference \cite[Lemma~3.9]{YaoConvexFiniteType} the second term in \eqref{Eqn::FiniteCX::BasisEst} is stated for $\big(\frac{\widehat Q\wedge(\dbar \widehat Q)^j}{\widehat S^{j+1}}\big)(z,\zeta)$. Nevertheless using $\big(\frac{\widehat Q\wedge(\dbar \widehat Q)^j}{\widehat S^{j+1}}\big)^\bot=\big(\frac{\widehat Q\wedge(\dbar \widehat Q)^j}{\widehat S^{j+1}}\big)-\big(\frac{\widehat Q\wedge(\dbar \widehat Q)^j}{\widehat S^{j+1}}\big)^\top$ and \ref{Eqn::TauOrder} we get the same estimate (with some larger $C_k$) for $\big(\frac{\widehat Q\wedge(\dbar \widehat Q)^j}{\widehat S^{j+1}}\big)^\bot(z,\zeta) $.

This Leray map was constructed by Diederich and Forn\ae ss \cite{DFSupport}. Note that in the original construction \cite{DFSupport} the support function $S(z,\zeta)$ may have zeros when $|z-\zeta|$ is large. In \cite[Lemma~2.2]{YaoConvexFiniteType} we took a standard modification to avoid the issue.

Now we can recall the homotopy operators $(\Hc_q)_{q=1}^n$ in \cite{YaoConvexFiniteType}, which takes the following form:
\begin{equation}\label{Eqn::SkewBerg::DefHq}
    \Hc_qf(z)=\int_\Uc B_{q-1}(z,\cdot)\wedge \Ec f+\int_{\Uc\backslash\overline\Omega}K_{q-1}(z,\cdot)\wedge[\dbar,\Ec]f,\qquad f\in\Ss'(\Omega;\wedge^{0,q}),\qquad 1\le q\le n.
\end{equation}
Here $\Uc$ is the bounded neighborhood of $\overline\Omega$ determined in Proposition~\ref{Prop::BasisEst}. $\Ec:\Ss'(\Omega)\to \Es'(\Uc)$ is Rychkov's extension operator \cite{RychkovExtension}, acting on the components of the forms, see \cite[(4.6) and (4.14)]{YaoSPsiCXC2} for the precise formula.
\begin{equation*}
    B(z,\zeta):=\frac{b\wedge(\dbar b)^{n-1}}{(2\pi i)^n|\zeta-z|^{2n}}=\sum_{q=0}^{n-1}B_q,\quad K(z,\zeta)=\frac{ b\wedge \widehat Q}{(2\pi i)^n}\wedge\sum_{k=1}^{n-1}\frac{(\dbar b)^{n-1-k}\wedge(\dbar\widehat Q)^{k-1}}{|z-\zeta|^{2(n-k)}(\widehat Q\cdot(\zeta-z))^k}=\sum_{q=0}^{n-2}K_q, 
\end{equation*}
where  $b=\sum_{j=1}^n(\bar\zeta_j-\bar z_j)d\zeta^j$. $B$ is the \textit{Bochner-Martinelli form}, with  $B_q$  the component of degree $(0,q)$   in $z$ and   $(n,n-q-1)$  in $\zeta$.
$K$ is the \textit{Leray-Koppelman form}  associated to  $\widehat Q(z,\zeta)$ in Proposition~\ref{Prop::BasisEst}, where $K_q$ is the component of degree $(0,q)$ in $z$ and   $(n,n-q-2)$   in $\zeta$.

\medskip

Denote by 
\begin{equation}\label{Eqn::SkewBerg::CF}
    F(z,\zeta):=B(z,\zeta)-\dbar_{z,\zeta}K(z,\zeta),\quad z\in\Omega,\quad \zeta\in\Uc\backslash\overline\Omega, 
\end{equation}
   the \textit{Cauchy-Fantappi\`e form}. Recall from \cite[Lemma~11.1.1]{ChenShawBook} we have
\begin{equation*}
    F(z,\zeta)=\frac{\widehat Q\wedge(\dbar\widehat Q)^{n-1}}{(2\pi i)^n\widehat S(z,\zeta)^n}=\frac{\widehat Q(z,\zeta)\wedge (\dbar\widehat Q(z,\zeta))^{n-1}}{(2\pi i)^n(\widehat Q(z,\zeta)\cdot(\zeta-z))^n},\quad z\in\Omega,\quad\zeta\in\Uc\backslash\overline\Omega.
\end{equation*}
Note that $F$ is a bi-degree $(n,n-1)$  form, with degree $(0,0)$ in $z$ and $(n,n-1)$ in $\zeta$.   We write the decomposition $F=F^\top+F^\bot$ in $\zeta$-variable as defined from above (see \cite[Convention~2.7]{YaoConvexFiniteType}).

\begin{prop}\label{Prop::WeiEst}
    Assume $\Omega$ is convex and has finite type $m$. Let $\delta(w):=\dist(w,b\Omega)$. Then for any $s\in\R$ and $k\in\Z_+$ such that $0<s<k-1$, there is a constant  $C=C(\Omega,\Uc,\widehat Q,k,s)>0$ such that
\begin{align}
    \label{Eqn::WeiEst::Topz}
    \int_{\Uc\backslash\overline\Omega}\delta(\zeta)^s |D^k_{z,\zeta}(F^\top)(z,\zeta)|d\Vol(\zeta)&\le C\delta(z)^{s+1-k},&&\forall z\in \Omega;
    \\
    \label{Eqn::WeiEst::Topzeta}
    \int_{\Omega}\delta(z)^s |D^k_{z,\zeta}(F^\top)(z,\zeta)|d\Vol(z)&\le C\delta(\zeta)^{s+1-k},&&\forall \zeta\in \Uc\backslash\overline\Omega;
    \\
    \label{Eqn::WeiEst::Botz}
    \int_{\Uc\backslash\overline\Omega}\delta(\zeta)^s|D^k_{z,\zeta}(F^\bot)(z,\zeta)|d\Vol(\zeta)&\le C\delta(z)^{s+\frac1m-k},&&\forall z\in \Omega;
    \\
    \label{Eqn::WeiEst::Botzeta}
    \int_{\Omega}\delta(z)^s |D^k_{z,\zeta}(F^\bot)(z,\zeta)|d\Vol(z)&\le C\delta(\zeta)^{s+\frac1m-k},&&\forall\zeta\in \Uc\backslash\overline\Omega.
\end{align}
As a result if we define for every $\alpha\in\N^{2n}_{\zeta,\bar\zeta}$
$$\Fc^{\alpha,\top}g(z):=\int_{\Uc\backslash\overline\Omega}D_\zeta^\alpha (F^\top)(z,\cdot)\wedge g,\quad \Fc^{\alpha,\bot}g(z):=\int_{\Uc\backslash\overline\Omega}D_\zeta^\alpha (F^\bot)(z,\cdot)\wedge g,\qquad g\in L^1(\Uc\backslash\overline\Omega; \wedge^{0,1}),$$
then in terms of Definition~\ref{Defn::Hsp}, for every $s>0$ and $1<p<\infty$, 
\begin{equation}\label{Eqn::FOpBdd}
    \Fc^{\alpha,\top}:\widetilde H^{s,p}(\overline\Uc\backslash\Omega;\wedge^{0,1})\to H^{s+1-|\alpha|,p}(\Omega),\qquad \Fc^{\alpha,\bot}:\widetilde H^{s,p}(\overline\Uc\backslash\Omega;\wedge^{0,1})\to H^{s+\frac1m-|\alpha|,p}(\Omega)
\end{equation}
are bounded.
\end{prop}
\begin{proof}
    Notice that for $0<\eps<\eps_0$ and $\zeta\in\Uc\backslash\Omega$ by \eqref{Eqn::SkewBerg::DefPEps} we have $|P_\eps(\zeta)\backslash P_{\eps/2}(\zeta)|\le |P_\eps(\zeta)| \le \prod_{l=1}^n\tau_l(\zeta,\eps)^2$. Similarly $|P_\eps(z)\backslash P_{\eps/2}(z)|\le \prod_{l=1}^n\tau_l(z,\eps)^2$ for $-\eps_0<\varrho(z)<0$  as well. 
    
According to Proposition \ref{Prop::BasisEst},  \eqref{Eqn::FiniteCX::BasisEst} holds for $0<\eps<\eps_0$, $\zeta\in\Uc\backslash\overline\Omega$ and $z\in \Omega\cap P_\eps(\zeta)\backslash P_{\frac\eps2}(\zeta)$. By \ref{Eqn::FlipP}, for a possibly larger $C_k>0$, one also has \eqref{Eqn::FiniteCX::BasisEst} holds for $0<\eps<\eps_0$, $-\eps_0<\varrho(z)<0$ and $\zeta\in P_\eps(z)\backslash( P_{\frac\eps2}(z)\cap\Omega)$. 
    
    Now take $j=n-1$ in \eqref{Eqn::FiniteCX::BasisEst}, for every $z\in\Omega$ and $\zeta\in\Uc\backslash\Omega$, we see that
    \begin{gather}
        \label{Eqn::WeiEst::BasisTop}
    \int_{\Omega \cap P_\eps(\zeta)\backslash P_{\frac\eps2}(\zeta)}|D^k(F^\top)(w,\zeta)|d\Vol(w)+\int_{P_\eps(z)\backslash (P_{\frac\eps2}(z)\cup \Omega)}|D^k(F^\top)(z,w)|d\Vol(w)\le C_k\eps^{1-k};
    \\
\label{Eqn::WeiEst::BasisBot}
    \int_{\Omega \cap P_\eps(\zeta)\backslash P_{\frac\eps2}(\zeta)}|D^k(F^\bot)(w,\zeta)|d\Vol(w)+\int_{P_\eps(z)\backslash(P_{\frac\eps2}(z)\cup \Omega)}|D^k(F^\bot)(z,w)|d\Vol(w)\le C_k\eps^{\frac1m-k}.
    \end{gather}

To prove  \eqref{Eqn::WeiEst::Topz}, we note that $0<s<k-1$, and $F$ is bounded  and smooth uniformly either for $z\in \Omega$ with $\delta(z)\ge \eps_0 $, or for $\zeta\in \Uc\backslash(P_{\eps_0}(z)\cup\Omega)$. 
 Thus it suffices to show $
     \int_{P_{\eps_0}(z)\backslash\overline\Omega}\delta(\zeta)^s|D^kF^\top(z,\zeta)|d\Vol_\zeta\lesssim \delta(z)^{s+1-k},\ \ \forall \delta(z)<\eps_0.
$
   Let $J\in\Z$ be the unique number such that $2^{-J}\eps_0\le\varrho(z)<2^{1-J}\eps_0$. Then  $P_{\eps_0}(z)\backslash\overline\Omega \subset \cup_{j=1}^J  P_{2^{1-j}\eps_0}(z)\backslash(P_{2^{-j}\eps_0}(z)\cup\Omega).$  Applying \eqref{Eqn::WeiEst::BasisTop} we get \eqref{Eqn::WeiEst::Topz}:
\begin{equation}\label{Eqn::WeiEst::PfTop}
    \begin{aligned}
    &\int_{P_{\eps_0}(z)\backslash\overline\Omega}\delta(\zeta)^s|D^kF^\top(z,\zeta)|d\Vol_\zeta\lesssim_{k}\sum_{j=1}^J\int_{P_{2^{1-j}\eps_0}(z)\backslash(P_{2^{-j}\eps_0}(z)\cup\Omega)}(2^{-j}\eps_0)^s|D^kF^\top(z,\zeta)|d\Vol_\zeta
    \\
    &\qquad\lesssim_k\sum_{j=1}^J(2^{-j}\eps_0)^s(2^{-j}\eps_0)^{1-k}\lesssim_{\eps_0}2^{-J(s+1-k)}\approx\delta(z)^{s+1-k}.
\end{aligned}
\end{equation}
By swapping $z$ and $\zeta$, the same argument yields \eqref{Eqn::WeiEst::Topzeta}.
Replacing \eqref{Eqn::WeiEst::BasisTop} by \eqref{Eqn::WeiEst::BasisBot}, the same computation as in \eqref{Eqn::WeiEst::PfTop} yields \eqref{Eqn::WeiEst::Botz} and \eqref{Eqn::WeiEst::Botzeta}.

The boundedness for $\Fc^{\alpha,\top}$ is a direct consequence from \cite[Corollary~A.28]{YaoSPsiCXC2} with \eqref{Eqn::WeiEst::Topz} and \eqref{Eqn::WeiEst::Topzeta}, similarly that of $\Fc^{\alpha,\bot}$ follows from \eqref{Eqn::WeiEst::Botz} and \eqref{Eqn::WeiEst::Botzeta}. The proof uses Hardy's distance inequality (see \cite[Proposition~5.3]{YaoConvexFiniteType}).
\end{proof}

\begin{prop}\label{Prop::Technical}
Let $\Omega\subset\C^n$ be a bounded smooth domain and $\Uc\supset\overline\Omega$ be a bounded smooth neighborhood. Let $\Ec$ be   Rychkov's extension operator in \cite[(4.14)]{YaoConvexFiniteType}. 
    \begin{enumerate}[(i)]
        \item\label{Item::Technical::Antidev} For every $k\ge1$ there are linear operators $(\Sc^{k,\alpha})_{|\alpha|\le k}:\Ss'(\C^n)\to\Ss'(\C^n)$ (here $\alpha\in\N^{2n}$) such that $\Sc^{k,\alpha}:\widetilde H^{s,p}(\overline{\Uc\backslash\Omega})\to\widetilde H^{s+k,p}(\overline{\Uc\backslash\Omega})$ is bounded and $g=\sum_{|\alpha|\le k}D^\alpha\Sc^{k,\alpha}g$ for all $\supp g\subset\Uc\backslash\Omega$.
        \item\label{Item::Technical::TangComm} The map $[f\mapsto([\dbar,\Ec]f)^\top]:  H^{s,p}(\Omega)\to \widetilde H^{s-\eps,p}(\overline{\Uc\backslash\Omega}) $ is bounded  for all $s\in\R$, $\eps>0$ and $1<p<\infty$.
    \end{enumerate}
\end{prop}
See \cite[Proposition~1.7]{ShiYaoExt} for \ref{Item::Technical::Antidev} and \cite[Corollary~5.5~(iii)]{YaoConvexFiniteType} for \ref{Item::Technical::TangComm}.

\begin{proof}[Proof of Theorem~\ref{Thm::PforCXFinite}]
    Since $B-F=\dbar_{z,\zeta}K$ by \eqref{Eqn::SkewBerg::CF}, by separating the degrees we see that $F=B_0-\dbar_\zeta K_0$. For the same extension operator $\Ec$ in \eqref{Eqn::SkewBerg::DefHq} we have (see e.g. \cite[Proposition~2.1]{GongHolderSPsiCXC2}), for $f\in\Ss'(\Omega)$,
    \begin{align*}
        \Pc f(z)=&f(z)-\Hc_1\dbar f(z)=\Ec f(z)-\int_\Uc B_0(z,\cdot)\wedge\Ec\dbar f-\int_{\Uc\backslash\overline\Omega} K_0(z,\cdot)\wedge[\dbar,\Ec]\dbar f
        \\
        =&\int_\Uc \dbar_\zeta B_0(z,\cdot)\wedge \Ec f-\int_\Uc B_0(z,\cdot)\wedge\dbar\Ec f+\int_{\Uc\backslash\overline\Omega} B_0(z,\cdot)\wedge[\dbar,\Ec] f+\int_{\Uc\backslash\overline\Omega} K_0(z,\cdot)\wedge\dbar[\dbar,\Ec] f
        \\
        =&\int_{\Uc\backslash\overline\Omega} B_0(z,\cdot)\wedge[\dbar,\Ec]f-\int_{\Uc\backslash\overline\Omega}\dbar_\zeta K_0(z,\cdot)\wedge[\dbar,\Ec]f=\int_{\Uc\backslash\overline\Omega}F(z,\cdot)\wedge[\dbar,\Ec]f.
    \end{align*}
    Fix $s\in\R$ and $1<p<\infty$. It suffices  to show $\Pc:H^{s,p}(\Omega)\to H^{s,p}(\Omega)$ is bounded.
    
    Take $k\in\Z_+$ such that $k>1-s$. By the $\top,\bot$ decomposition, Proposition~\ref{Prop::Technical}~\ref{Item::Technical::Antidev} and integration by parts we have
    \begin{align*}
        \Pc f(z)=&\int_{\Uc\backslash\overline\Omega}F^\top(z,\cdot)\wedge([\dbar,\Ec]f)^\bot+F^\bot(z,\cdot)\wedge([\dbar,\Ec]f)^\top 
        \\
        =&\sum_{|\alpha|\le k}\int_{\Uc\backslash\overline\Omega}F^\top(z,\cdot)\wedge D^\alpha\Sc^{k,\alpha}\big[([\dbar,\Ec]f)^\bot\big]+F^\bot(z,\cdot)\wedge D^\alpha\Sc^{k,\alpha}\big[([\dbar,\Ec]f)^\top\big]          \\
        =&\sum_{|\alpha|\le k}(-1)^{|\alpha|}\int_{\Uc\backslash\overline\Omega}D_\zeta^\alpha( F^\top)(z,\cdot)\wedge\Sc^{k,\alpha}\big[([\dbar,\Ec]f)^\bot\big]+D_\zeta^\alpha (F^\bot)(z,\cdot)\wedge\Sc^{k,\alpha}\big[([\dbar,\Ec]f)^\top\big]
        \\
        =&\sum_{|\alpha|\le k}(-1)^{|\alpha|}\Big(\Fc^{\alpha,\top}\Sc^{k,\alpha}[\dbar,\Ec]^\bot+\Fc^{\alpha,\top}\Sc^{k,\alpha}[\dbar,\Ec]^\top\Big)[f].
                 \end{align*}
    Here we use $[\dbar,\Ec]^{(\top,\bot)} f:=([\dbar,\Ec] f)^{(\top,\bot)}$.

    Note that by Proposition~\ref{Prop::Technical}~\ref{Item::Technical::TangComm} $[\dbar,\Ec]^\top:H^{s,p}(\Omega)\to\widetilde H^{s-1/m,p}(\overline\Omega)$ is bounded. On the other hand, since  $[\dbar,\Ec]:H^{s,p}(\Omega)\to\widetilde H^{s-1,p}(\overline{\Uc\backslash \Omega})$ is clearly bounded  and   $[\dbar,\Ec]^\bot=[\dbar,\Ec]-[\dbar,\Ec]^\top$, we have the boundedness  $[\dbar,\Ec]^\bot: H^{s,p}(\Omega)\to\widetilde H^{s-1,p}(\overline{\Uc\backslash \Omega})$. 
    Making use of those, together with the boundedness for $\Sc^{k,\alpha} $ in Proposition \ref{Prop::Technical}~\ref{Item::Technical::Antidev}, as well as   for $\Fc^{\alpha,\top}$ and $\Fc^{\alpha,\bot} $ in \eqref{Eqn::FOpBdd}, we  apply the following composition arguments:  for every $|\alpha|\le k$
    \begin{align*}
        \Fc^{\alpha,\top}\Sc^{k,\alpha}[\dbar,\Ec]^\bot:&H^{s,p}(\Omega)\xrightarrow{[\dbar,\Ec]^\bot}\widetilde H^{s-1,p}(\overline{\Uc\backslash\Omega})\xrightarrow{\Sc^{k,\alpha}}\widetilde H^{s-1+k,p}(\overline{\Uc\backslash\Omega})\xrightarrow
        {\Fc^{\alpha,\top}}H^{s,p}(\Omega);
        \\
        \Fc^{\alpha,\bot}\Sc^{k,\alpha}[\dbar,\Ec]^\top:&H^{s,p}(\Omega)\xrightarrow{[\dbar,\Ec]^\top}\widetilde H^{s-1/m,p}(\overline{\Uc\backslash\Omega})\xrightarrow{\Sc^{k,\alpha}}\widetilde H^{s-1/m+k,p}(\overline{\Uc\backslash\Omega})\xrightarrow
        {\Fc^{\alpha,\bot}}H^{s,p}(\Omega).
    \end{align*}
    Taking sums over $\alpha$ we complete the proof.
\end{proof}

\begin{ack}
    The authors would like to thank Song-Ying Li for some helpful discussion.
\end{ack}
\bibliographystyle{amsalpha}
\bibliography{references} 
\end{document}